\documentclass{amsart}

\usepackage{amsmath}
\usepackage{amssymb}
\usepackage{amsthm}
\usepackage[a4paper]{geometry}

\usepackage{tikz}
\usepackage{tikz-cd}
\usepackage{tikz-3dplot}
\usetikzlibrary{shapes.geometric}
\usetikzlibrary{intersections}
\usetikzlibrary{calc,angles,positioning,intersections,quotes,decorations.markings}

\usepackage{mathrsfs}

\usepackage[
backend=biber,
style=alphabetic,
maxalphanames = 4,
maxbibnames = 99,
sorting=ynt
]{biblatex}
\usepackage{subcaption} 

\usepackage{hyperref}

\usepackage{subfiles}

\theoremstyle{plain}
\newtheorem{definition}{Definition}[section]
\newtheorem{theorem}[definition]{Theorem}
\newtheorem{lemma}[definition]{Lemma}

\theoremstyle{definition}

\newcommand{\Z}{\mathbb{Z}}
\newcommand{\R}{\mathbb{R}}

\newcommand{\G}{\mathcal{G}} 
\newcommand{\B}{\mathcal{B}} 

\newcommand{\A}{\mathcal{A}}
\newcommand{\RP}{\mathbb{RP}}

\newcommand{\GG}{\mathcal{G}}
\newcommand{\Swap}{\mathcal{S}}
\newcommand{\sln}{\mathfrak{sl}}

\DeclareMathOperator{\PSL}{PSL}

\DeclareMathOperator{\Tr}{Tr}
\DeclareMathOperator{\Id}{Id}

\newcommand{\comment}[1]{}
\DeclareMathOperator{\Hom}{Hom}

\title{The Structure of the Ghost and Swapping Algebras}
\author{Caleb Ashley, Ming Hong Tee}
\address{Caleb Ashley, Department of Mathematics, Boston College,
Chestnut Hill, Ma 02467}
\thanks{Ashley's research is supported by NSF grant DMS-2316685}
\address{Ming Hong Tee, Department of Mathematics, Boston College,
Chestnut Hill, Ma 02467}
\date{}

\begin{document}
\maketitle
\begin{abstract}
    The swapping algebras and the ghost algebras were introduced in the papers of Labourie and Bridgeman--Labourie to encode the Poisson structure on spaces of length functions and generalized cross-ratio functions associated to $\Theta$-Anosov representations. They used these algebras to show that correlation functions are stable under the Poisson bracket associated to Atiyah-Bott-Goldman's symplectic form and that length functions are convex along the Hamiltonian flow of other length functions. In this paper, we study the algebraic relationship between these two algebras. In particular, we show that the kernel of a natural surjection from the ghost algebra onto the extended swapping algebra is generated by five families of relations. Consequently, we obtain an isomorphism from an explicit quotient of the ghost algebra onto the extended swapping algebra.
\end{abstract}

\section{Introduction}
The swapping algebra and the ghost algebra were introduced in \cite{labourie2018goldman,bridgeman2023ghost}, respectively. Both algebras are utilized in the study of families of representations of discrete groups into Lie groups. Many natural functions arise on spaces of such representations, for example, in the case of surface groups, the lengths of closed geodesics or the traces of holonomies, which have Poisson brackets given by combinatorial formulas involving the intersections of curves (see Goldman \cite{Goldman1986}). The swapping algebra and ghost algebra encode such formulas in higher rank. In this section, we provide context for these algebras and thus situate our main result (Theorem \ref{thm: relationship between G and S}).  

\subsection{Background}
Given a closed Riemann surface $S$ of genus $g \ge 2$, the Teichm\"{u}ller space of $S$, consisting of marked complex structures on $S$, corresponds via uniformization, to the space of discrete faithful representations of $\pi_1(S)$, the fundamental group of $S$, into the real projective special linear group $\PSL(2,\R)$, up to inner automorphism. Work of Goldman \cite{Goldman1988} identifies the Teichm\"{u}ller space of $S$ with a particular connected component of the character variety $\Hom(\pi_1(S), \PSL(2,\R))/\PSL(2,\R)$. More generally, one can consider the character variety of $\pi_1(S)$ into a higher-rank Lie group $G$, such as $\PSL(n,\R)$ for $n\geq 3$. Analyzing to what extent elements of the classical theory of $\PSL(2,\R)$  persist, and what new phenomena arise in higher rank, constitutes what is known as higher Teichm\"{u}ller theory.

Goldman \cite{goldman1984symplectic} introduced a natural symplectic form on the smooth loci of varieties associated to representations of surface groups into a reductive Lie group $G$. The symplectic form is used to study the geometry of the corresponding deformation spaces. Furthermore, he showed that in the case of $\PSL(2,\R)$, the restriction of this symplectic form to the Teichm\"{u}ller component is in fact a constant multiple of the Weil--Petersson symplectic form.

Given a symplectic form on a manifold, it is natural to compute the Hamiltonian flows of real-valued smooth functions. In the case of Teichm\"{u}ller space, an important class of functions consists of the length and twist functions of curves on the surface. Gardiner \cite{gardiner1975schiffer} and Wolpert \cite{wolpert1982fenchel,wolpert1983symplectic,wolpert1985weil} showed that the Hamiltonian flow of the length function of an essential simple closed curve corresponds to the infinitesimal twist vector field along that curve.

One might ask what the Hamiltonian flows of important families of functions are when $\PSL(2,\R)$ is replaced by a higher-rank Lie group. To make sense of this question, we first need a space on which these functions are defined. A well-known example is the Hitchin component \cite{Hitchin1992}. For $\PSL(n,\R)$, it is the connected component of the character variety that contains the composition of a Fuchsian representation from $\pi_1(S)\to \PSL(2,\R)$ with the irreducible representation $\PSL(2,\R)\to\PSL(n,\R)$. When $n=3$, Goldman \cite{Goldman1990} and Choi--Goldman \cite{ChoiGoldman1993} verified that the Hitchin component parameterizes the convex real projective structures on $S$. Labourie \cite{labourie2006anosov} introduced Anosov representations and showed that Hitchin representations are Anosov. Following that, Guichard and Wienhard \cite{guichard_wienhard_2012} extended this notion to representations of word hyperbolic groups into a semisimple Lie groups. These representations are now called $\Theta$-Anosov representations. The space we consider lies in a more general setting. It is the universal $\Theta$-Anosov space of Bridgeman and Labourie \cite{bridgeman2023ghost}, which contains the spaces of $\Theta$-Anosov representations of all closed surfaces.

For $\Theta$-Anosov representations of closed surface groups, as with the classical Teichm\"{u}ller space, important classes of functions are the $\Theta$-length functions. These functions are essential for defining the pressure form \cite{bcls2015pressure} (a generalization of the Weil--Petersson metric) and for establishing marked length spectrum rigidity results \cite{bcls2015pressure,BridgemanCanaryLabourie2020}. Given such length functions, Bridgeman and Labourie \cite{bridgeman2023ghost} showed how to compute their Poisson brackets on the space of Anosov representations. In fact, analogously to the universal Teichm\"{u}ller space, they introduced uniformly hyperbolic bundles, which play the role of a universal space for the Anosov representations of all closed surface groups. Length functions are not defined in this universal setting; however, correlation functions, which are closely related to length functions, are.   

To explain the notion of correlation functions and how they are related to length functions, we examine the case of $\Theta$-Anosov representations of a surface group $\pi_1(S)$ into $\PSL(n,\R)$, where $\Theta = \{1\}$. In this case, these representations are also called projective Anosov representations. For each nontrivial $\gamma \in \pi_1(S)$, the projective length is given by
\begin{align*}
    L_\gamma(\rho) = \log(\Lambda_1(A)),
\end{align*}
where $\rho$ is a projective Anosov representation into $\PSL(n,\R)$, $A = \rho(\gamma)$, and $\Lambda_1(A)$ is the largest modulus of an eigenvalue of $A$. Note that this definition depends on the surface group. We compute the derivative of $L_\gamma$ along a smooth path of projective Anosov representations $\rho_s$ with $\rho_0 = \rho$. Let $X = A^{-1}A' \in \sln(n,\R)$, where $A' = \frac{d}{ds}\rho_s(\gamma)|_{s=0}$ and $\rho' = \frac{d}{ds}\rho_s|_{s=0}$. Then, a computation gives; 
\begin{align} \label{eqn: derivative of length}
    (dL_\gamma)_\rho(\rho') = \Tr(P_1(A)X) = \Tr\left(\left(P_1(A)-\frac{\Id}{n}\right)X\right),
\end{align}
where $P_1(A)$ is the projector whose image is the eigenline $E$ of $A$ for the eigenvalue with largest modulus, and whose kernel is the sum $E^*$ of the generalized eigenspaces for the remaining eigenvalues. Since the limit maps corresponding to $\rho$ are maps, 
$$(\zeta,\chi):\partial_\infty \pi_1(S) = S^1 \to \RP^{n-1}\times (\RP^{n-1})^*,$$ 
with the properties that $\zeta(\gamma^+) = E$ and $\chi(\gamma^-) = E^*$, where $\gamma^{\pm}$ are the attracting and repelling fixed points of $\gamma$ in $\partial_\infty \pi_1(S)$, the projector $P_1(A)$, and hence the derivative of $L_\gamma$, can be written in terms of the limit maps.

Equation (\ref{eqn: derivative of length}) shows that the function $P_1(A)-\frac{\Id}{n}$ is dual to the differential of $L_\gamma$ under the trace form. Therefore, for $\gamma'\in \pi_1(S)$, the variation of $L_\gamma$ along the dual of $L_{\gamma'}$ is
\begin{align} \label{eqn: bracket}
    \Tr\left(\left(P_1(A)-\frac{\Id}{n}\right)\left(P_1(B)-\frac{\Id}{n}\right)\right) = \Tr\left(P_1(A)P_1(B)\right) - \frac{1}{n},
\end{align}
where $B= \rho(\gamma')$. Hence, the derivative of $L_\gamma$ along the dual of $L_\gamma'$ is related to the function $\Tr\left(P_1(A)P_1(B)\right)$. The function $\Tr\left(P_1(A)P_1(B)\right)$ is an example of a correlation function, which in general is defined as the trace of a product of projectors \cite{labourie2018goldman,bridgeman2023ghost}. For a general correlation function $T_G = \Tr(P_1(A_1)\dots P_1(A_k))$, with $A_i = \rho(\gamma_i)$, the variation gives
\begin{align} \label{eqn: variation of correlation function}
    T_G' = \sum_{i=1}^k \Tr(P_1(A_1)\dots P_1'(A_i)\dots P_1(A_k)),
\end{align}
which shows that the variation of $T_G$ is determined by the variations of projectors. Bridgeman and Labourie gave a formal description of the variation of projectors. With it they reorganized Equation (\ref{eqn: variation of correlation function}) and called the outcome ''ghost integration" along the configuration $G=\lceil g_1,\dots,g_k\rceil$ where $g_i$ corresponds to a lift of $\gamma_i$. Furthermore, they defined the variation of a correlation function along a vector field associated to a projector or a correlation function. This procedure is encoded by the ghost algebra.

Uniformly hyperbolic bundles are trivial vector bundles over the unit tangent bundle of the hyperbolic plane, equipped with a connection (which gives a flow) and a metric, that satisfy conditions similar to those defining Anosov representations \cite{bridgeman2023ghost}. These bundles can be defined without reference to a group and come with limit maps. Hence, every oriented geodesic $g$ in the hyperbolic plane determines a projector, with image $\zeta(g^+)$ and kernel $\chi(g^-)$, and therefore correlation functions are defined in this setting.

In the periodic case, by comparing with Goldman's computation of Hamiltonian flows of invariant functions \cite{Goldman1986}, Bridgeman and Labourie computed the variation of length or correlation function along the Hamiltonian flow of another such function, using the tools above \cite{bridgeman2023ghost}. In particular, they showed the Poisson stability of the class of length and correlation functions on the spaces of Anosov representations of closed surface groups, i.e. the Poisson bracket of two functions that are length or correlation functions is a combination of correlation functions. Also, in analogy with Kerckhoff's convexity theorem \cite{kerckhoff1983nielsen} and Wolpert's sine formula for length functions in classical Teichm\"{u}ller theory \cite{wolpert1983symplectic}, they established the convexity of length functions along the Hamiltonian flows of length functions for positively ratioed projective Anosov representations, which were introduced in \cite{martone2019positively}.

The purpose of the ghost algebra is to compute the variation of correlation functions in the universal setting of uniformly hyperbolic bundles. This leads to a method for computing the Poisson brackets of such functions in the periodic case and leads to the proofs of the two stated results. For the second result, in the projective case, they showed that the ghost algebra maps onto an extension of the swapping algebra, a Poisson algebra first introduced in \cite{labourie2018goldman}, and used the swapping algebra to directly compute the second derivatives of length functions.

A natural question is the following: what is the relationship between these two algebras? In the projective case, Bridgeman and Labourie showed that there is a surjective homomorphism $\pi$ from the ghost algebra onto the extended swapping algebra that preserves the brackets. Also, they remarked that the bilinear bracket in the ghost algebra does not satisfy the Jacobi identity in general, and hence is not a Poisson algebra. In this paper, we establish explicit relations in the ghost algebra that generate the kernel of $\pi$.

\subsection{Statement of the main result}
We start by introducing the two central objects of this paper. Let $\A$ be the real polynomial algebra generated by $\langle x,y \rangle \in S^1\times S^1$ quotient by the ideal generated by $\langle x,x \rangle$ for all $x\in S^1$. Let $\GG = S^1\times S^1-\Delta$, which represents the space of oriented geodesics $g=\langle g^+,g^-\rangle$ in the hyperbolic plane. If $g_1,g_2 \in \GG$ and $g_1$ intersects $g_2$ transversely, we define $\epsilon(g_1,g_2)$ to be their sign intersection. This function $\epsilon$ then extends to arbitrary pairs of elements of $S^1\times S^1$, for instance pair that intersect only at their endpoints, disjoint or pairs involving a degenerate geodesic $\langle x,x\rangle$ (see \cite{bridgeman2023ghost} for more details). The algebra $\A$, equipped with the bracket,
$$[\langle x,y \rangle, \langle z,w \rangle] = \epsilon(\langle z,w \rangle,\langle x,y \rangle)\langle x,w \rangle\langle z,y \rangle,$$
forms a Poisson algebra, which we call the \emph{swapping algebra} \cite{labourie2018goldman}. Since $\A$ is isomorphic as a ring to $\R[\GG]$ and $\R[\GG]$ is an integral domain, we can define $F(\A)$ to be the field of fractions of $\A$. For $\vec{Z} = (Z_1,\dots,Z_n) \in (S^1)^n$, $\vec{z} = (z_1,\dots,z_n) \in (S^1)^n$ and $\sigma\in S_n$, with $Z_i \neq z_i$ for all $i$, the \emph{multifraction} $[\vec{Z},\vec{z};\sigma]$ is the element
\begin{align*}
    [\vec{Z},\vec{z};\sigma] = \prod_{i=1}^n\frac{\langle Z_i,z_{\sigma(i)}\rangle}{\langle Z_i,z_i \rangle} \in F(\A).
\end{align*}
Let $\A^*$ be the subalgebra of $F(\A)$ generated by the multifractions. The \emph{extended swapping algebra} $\Swap$ is the real polynomial algebra generated by $c, l_g$ and $\A^*$ for all $g\in \GG$. Here, $c$ is an additional formal generator and is declared to be central for the bracket. If we set $L_\GG = \{l_g\mid g\in \GG\}$, then we have $\Swap = \A^*[L_\GG,c]$ is an $\R[\GG,c]$-module, where $g\in \GG$ acts by multiplication by $l_g$. The bracket on $\A^*$ extends to $\Swap$ by the Leibniz rule and
\begin{align*}
    [l_g,l_h] &= \frac{1}{gh}[g,h] - \epsilon(h,g)c,  \\
    [l_g,H] &= \frac{[g,H]}{g}, \\
    [c,\cdot] &= 0,
\end{align*}
for $H \in \A^*$.

For all $k\geq2$, let $C_k$ be the cyclic group of order $k$ and $\hat{\GG}^k = (\GG^k-\Delta_k)/C_k$, where $\Delta_k = \{(g,\dots,g)\in\GG^k\}$ and $C_k$ acts by cyclically permuting the indices. We denote an element of $\hat{\GG}^k$ by $\lceil g_1,\dots,g_k \rceil$ and set $\hat{\GG}^* = \bigcup_{k\geq 2} \hat{\GG}^k$. For each $G\in \hat{\GG}^*$, one associates a ghost polygon, which is a collection of elements of $S^1\times S^1$ (see \cite{bridgeman2023ghost}). The \emph{ghost algebra} $\B$ is the real polynomial algebra generated by $\hat{\GG}^*$, $\GG$ and $c$, where $c$ is again a central formal element. The ghost bracket is a bilinear operator satisfying
\begin{align*}
    [g,h] &= \epsilon(h,g)(\lceil g,h\rceil - c), \\
    [g,H] &= \sum_{i} (-1)^{i+1} \epsilon(\tau_i,g) \lceil g,\tau_i^* \rceil, \\
    [G,H] &= \sum_{i,j} (-1)^{i+j} \epsilon(\tau_i,\theta_j) \lceil \theta_j^*,\tau_i^* \rceil, \\
    [c,\cdot] &= 0,
\end{align*}
where $G, H\in \hat\GG^*$, $\theta_j$ and $\tau_i$ are the elements of the ghost polygons associated to $G$ and $H$, respectively, and $\theta_j^*,\tau_i^*$ are the dual tuples of geodesics (see \cite{bridgeman2023ghost}).

The map $\pi:\B \to \Swap$ defined by $c\mapsto c$, $g\mapsto l_g$, and
$$\lceil g_1,\dots,g_n\rceil \mapsto \prod_{i=1}^n \frac{\langle g_i^+,g_{i+1}^- \rangle}{\langle g_i^+,g_i^-\rangle},$$
is an algebra map that preserves the bracket \cite{bridgeman2023ghost}.

To ease the notation in the proofs, we choose to represent elements of $\hat{\G}^*$ as $\lceil \langle X,Y \rangle,\dots, \langle S,T\rangle  \rceil$. Consider the ideal $I$ in $\B$ generated by the following elements:
\begin{enumerate}
    \item $\lceil \langle X,Y\rangle,\langle T,X\rangle\rceil$ \label{rel 1}
    \item $\lceil \langle X,Y\rangle,\langle X,S \rangle\rceil -1$ \label{rel 2}
    \item $\lceil \langle X,Y\rangle,\dots,\langle X,S\rangle  \rceil - \lceil \dots,\langle X,S\rangle\rceil $ \label{rel 3}
    \item $\lceil \langle X,Y\rangle,\langle T,Y\rangle, \dots  \rceil - \lceil \langle T,Y\rangle,\dots \rceil $ \label{rel 4}
    \item $\lceil \langle X,Y\rangle,\dots\rceil \lceil \langle X,S\rangle,\langle T,Y \rangle \rceil - \lceil \langle X,S\rangle,\dots,\langle T,Y \rangle\rceil$ \label{rel 5}
\end{enumerate}
where $``\dots"$ represents at least one element of $\G$, and distinct letters represent distinct elements of $S^1$.
   
\begin{theorem} \label{thm: relationship between G and S}
    The kernel of $\pi$ is the ideal $I$. In particular, $\B/I$ is isomorphic to $\Swap$.
\end{theorem}

\subsection{Acknowledgments}
Both authors thank Professor Bridgeman for introducing them to this mathematical work and for many enlightening discussions. 

\section{Ideas for the Proof}
\subsection{Heuristic for the generators of $I$}
We begin with a heuristic explanation of the generators of $I$. Recall that $\Swap$ contains elements of the form $[\vec{Z},\vec{z};\sigma]$, where $\vec{Z},\vec{z} \in (S^1)^n$ and $\sigma\in S_n$. To such an element $[\vec{Z},\vec{z};\sigma]$, we associate a labeled directed graph with vertex set $\{Z_1,\dots,Z_n\}\cup\{z_1,\dots,z_n\} \subset S^1$, directed visible edges $\{(Z_i,z_i)\}_{i=1}^n$ with label $1$, and directed ghost edges $\{(Z_i,z_{\sigma(i)})\}_{i=1}^n$ with label $-1$. If some ghost edge $(Z_i,z_{\sigma(i)})$ has $Z_i = z_{\sigma(i)}$, we assign the empty graph, which corresponds to the zero element of $\Swap$ (this comes from the relation $\langle x,x\rangle = 0$ in the definition of $\A$). This explains the first relation, that is, the first family of generators of $I$ (see Figure \ref{fig:rel 1}):
\begin{align*}
    \pi(\lceil\langle X,Y\rangle,\langle T,X\rangle\rceil) = \frac{\langle X,X\rangle\langle T,Y\rangle}{\langle X,Y\rangle\langle T,X\rangle} = 0.
\end{align*}
\begin{figure}[htbp]
    \centering
    \begin{subfigure}[b]{0.45\textwidth}
        \centering
        \includegraphics[width=\textwidth]{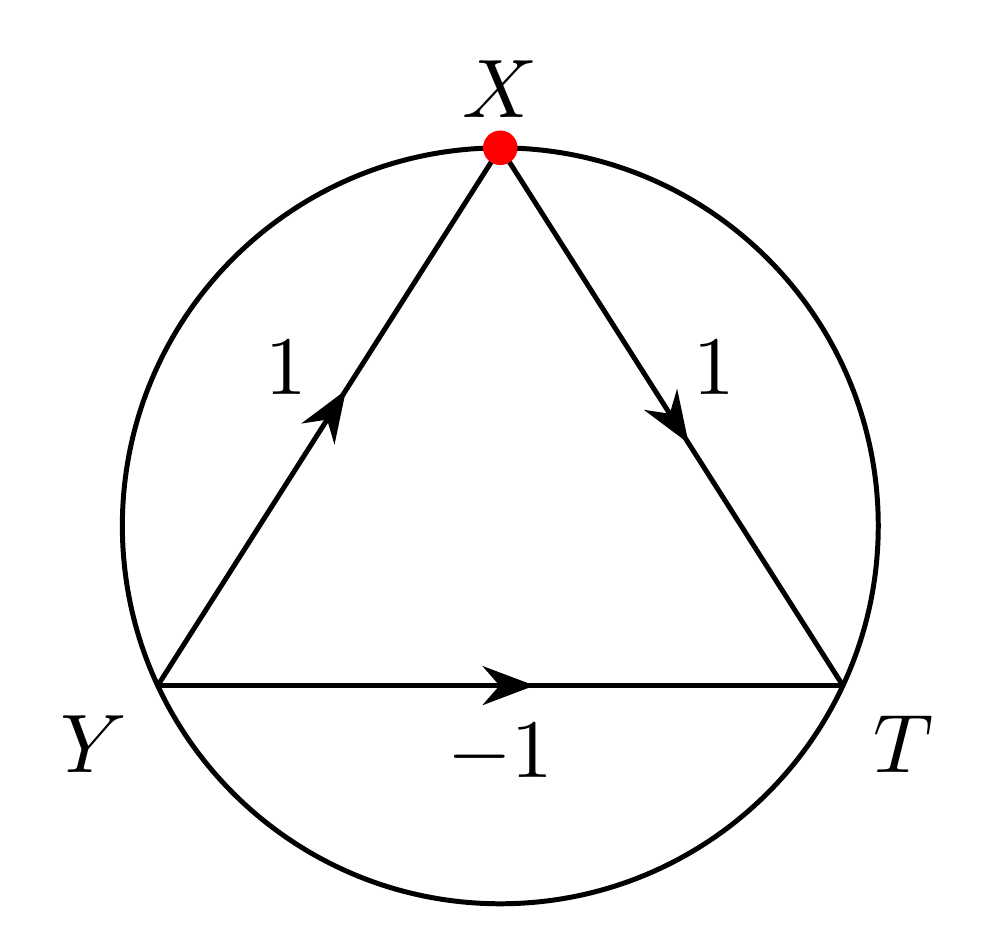}
        \caption{Relation 1}
        \label{fig:rel 1}
    \end{subfigure}
    \hfill 
    \begin{subfigure}[b]{0.45\textwidth}
        \centering
        \includegraphics[width=\textwidth]{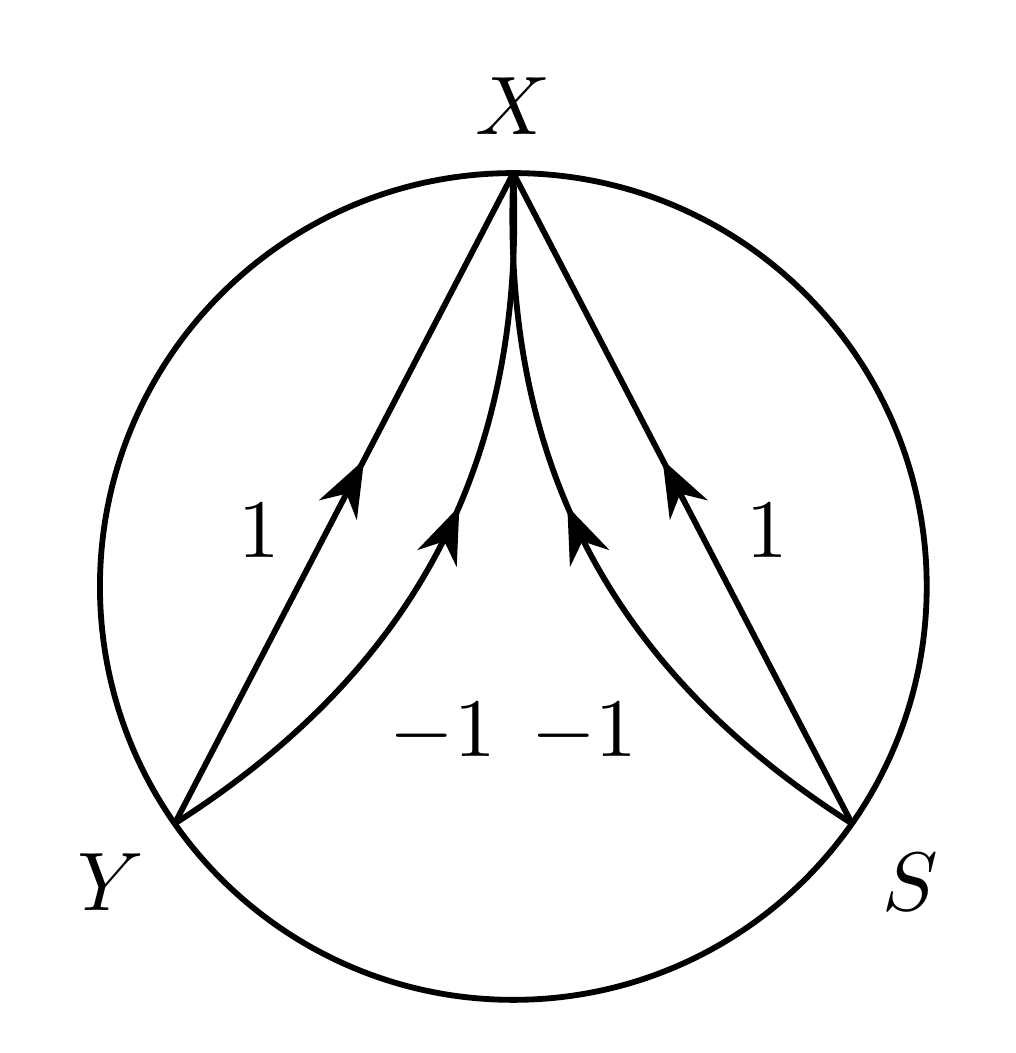}
        \caption{Relation 2}
        \label{fig:rel 2}
    \end{subfigure}
   
    \caption{The graphs associated to the generators of relations 1 and 2. For the graph on the left, the ghost edge $\langle X,X\rangle$ degenerates to the red point.}
    \label{fig:rel 1 and 2}
\end{figure}

If every ghost edge has distinct endpoints, we merge identical directed edges together and relabel each merged edge by the sum of their labels. This operation reflects a property of the field of fractions $F(\A)$: any two elements $[\vec{Z},\vec{z};\sigma]$ and $[\vec{W},\vec{w};\tau]$ with the same resulting graph are equal in $\Swap$. Relation \ref{rel 2}, illustrated in Figure \ref{fig:rel 2}, is constructed based on this collapsing property: both pairs of edges cancel, and 
\begin{align*}
    \pi(\lceil\langle X,Y\rangle,\langle X,S\rangle\rceil) = \frac{\langle X,S\rangle\langle X,Y\rangle}{\langle X,Y\rangle\langle X,S\rangle} = 1.
\end{align*}

For the cancellation relations, namely relations $3$, $4$ and $5$, Figure \ref{fig:cancellation rel} shows that the corresponding graphs collapse. Furthermore, we see that relations $3$ and $4$ are ``mirror images" of one another. By the definition of $\pi$, the cancellations in Figure \ref{fig:cancellation rel} read
\begin{align*}
    \pi(\lceil \langle X,Y\rangle,\langle U,T\rangle,\langle X,S\rangle\rceil) &= \frac{\langle X,T\rangle\langle U,S\rangle\langle X,Y\rangle}{\langle X,Y\rangle\langle U,T\rangle\langle X,S\rangle} = \frac{\langle X,T\rangle\langle U,S\rangle}{\langle U,T\rangle\langle X,S\rangle} = \pi(\lceil \langle U,T\rangle,\langle X,S\rangle\rceil), \\
    \pi(\lceil \langle X,Y\rangle,\langle T,Y\rangle,\langle U,S\rangle\rceil) &= 
    \frac{\langle X,Y\rangle\langle T,S\rangle\langle U,Y\rangle}{\langle X,Y\rangle\langle T,Y\rangle\langle U,S\rangle}=
    \frac{\langle T,S\rangle\langle U,Y\rangle}{\langle T,Y\rangle\langle U,S\rangle} =
    \pi(\lceil \langle T,Y\rangle,\langle U,S\rangle\rceil),
\end{align*}
and 
\begin{align*}
    \pi(\lceil \langle X,Y\rangle,\langle B,A\rangle\rceil\lceil \langle X,S\rangle,\langle T,Y\rangle\rceil) &= 
    \frac{\langle X,A\rangle\langle B,Y\rangle}{\langle X,Y\rangle\langle B,A\rangle}
    \frac{\langle X,Y\rangle\langle T,S\rangle}{\langle X,S\rangle\langle T,Y\rangle} \\
    &= 
    \frac{\langle X,A\rangle\langle B,Y\rangle\langle T,S\rangle}{\langle X,S\rangle\langle B,A\rangle\langle T,Y\rangle}=
    \pi(\lceil \langle X,S\rangle,\langle B,A\rangle,\langle T,Y \rangle\rceil).
\end{align*}
The same cancellations show that every generator of type \ref{rel 3}, \ref{rel 4} or \ref{rel 5} lies in $\ker \pi$.
\begin{figure}[htbp]
    \centering
    \begin{subfigure}[b]{0.45\textwidth}
        \centering
        \includegraphics[width=\textwidth]{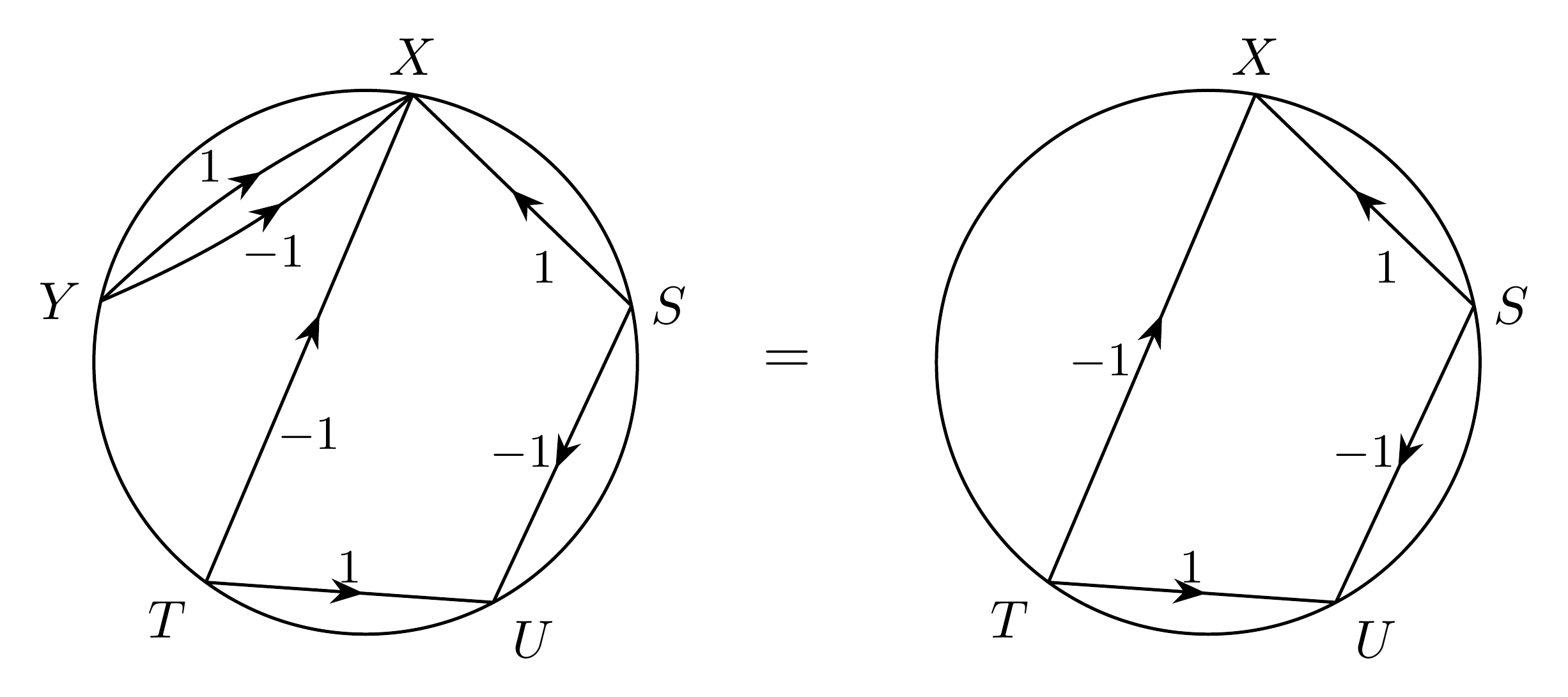}
        \caption{Relation 3}
        \label{fig:rel 3}
    \end{subfigure}
    \hfill 
    \begin{subfigure}[b]{0.45\textwidth}
        \centering
        \includegraphics[width=\textwidth]{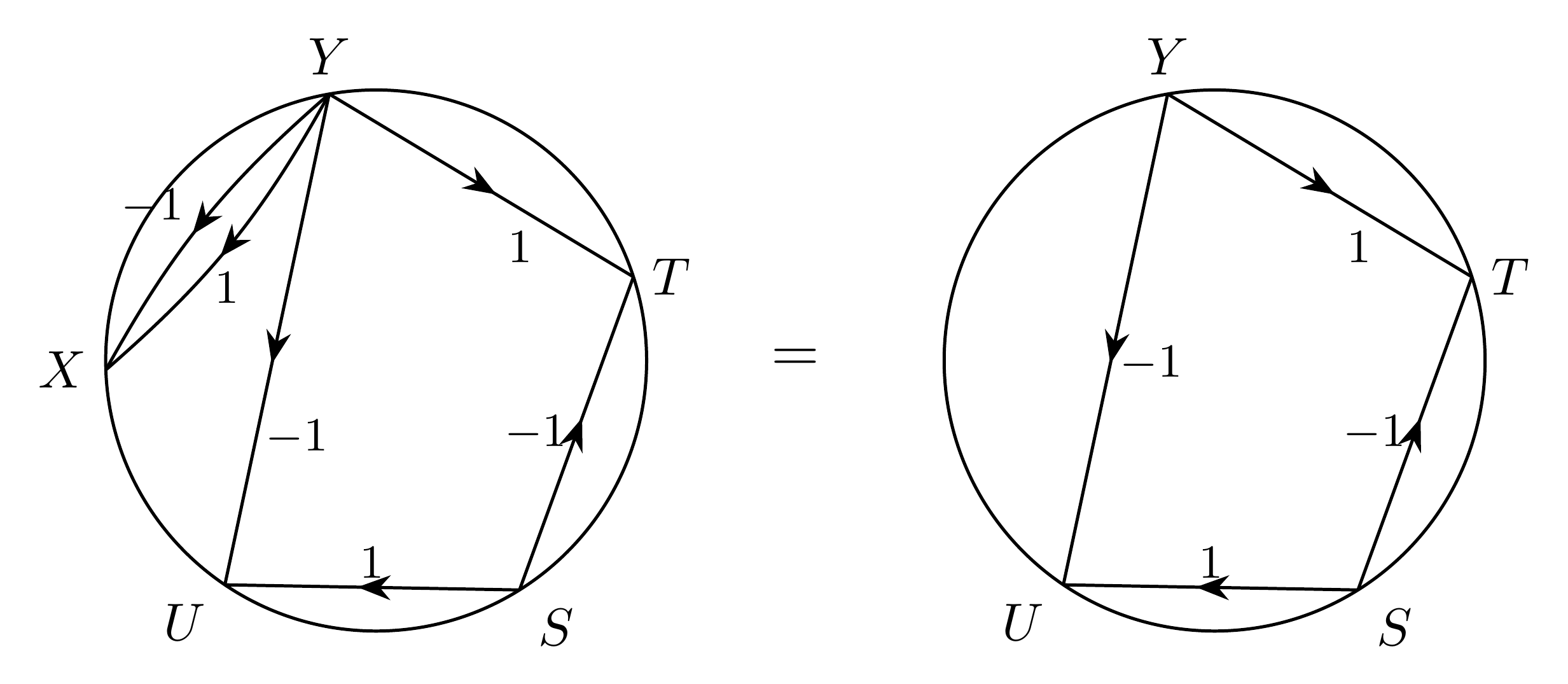}
        \caption{Relation 4}
        \label{fig:rel 4}
    \end{subfigure}
   
    \par\bigskip 
   
    \begin{subfigure}[b]{0.45\textwidth}
        \centering
        \includegraphics[width=\textwidth]{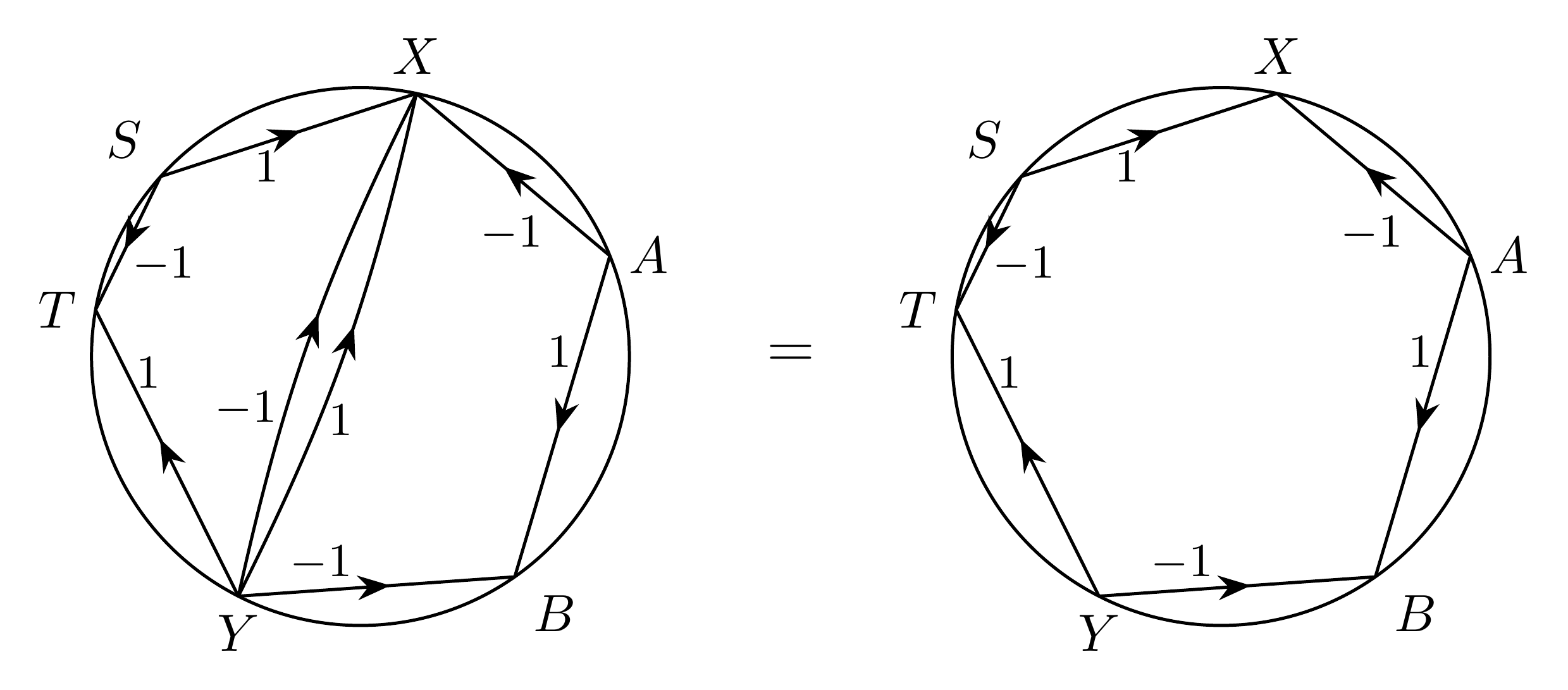}
        \caption{Relation 5}
        \label{fig:rel 5}
    \end{subfigure}
   
    \caption{The graph associated to the generators of relations $3$, $4$ and $5$.}
    \label{fig:cancellation rel}
\end{figure}

\subsection{Strategy of the proof} \label{subsection: strategy}
We now describe the structure of the proof. Note that $I\subseteq \ker \pi$, which induces a surjective map $\hat\pi:\B/I\to \Swap$. The main aim is to construct a well-defined $\R$-algebra map $\Psi$ from $\Swap$ to $\B/I$ that is the inverse of $\hat\pi$. In Lemma \ref{lem: basis for Swap},  we consider an $\R[\GG,c]$-basis $T$ for $\Swap$, consisting of the nonzero multifractions (Lemma \ref{lem: ppty of multifraction}). As with rational numbers, we represent each element of $T$ in lowest terms (equivalently, by its reduced labeled directed graph). Since $T$ is a basis, any map from $T$ to $\B/I$ extends uniquely to an $\R[\GG,c]$-module map from $\Swap$ to $\B/I$.

For a cross-ratio $\frac{\langle X,T\rangle\langle S,Y\rangle}{\langle X,Y\rangle\langle S,T\rangle} \in \Swap,$ the corresponding element in $\B/I$ is $\lceil \langle X,Y\rangle,\langle S,T\rangle \rceil$. To construct the required map $\psi:T\to \B/I$ in general, we consider in Lemma \ref{lem:presentation of T} the cross-ratio representation of elements in $T$, that is, a canonical decomposition of each element in $T$ into a product of cross-ratios (see Figure \ref{fig:diagram_decomp}). However, this decomposition depends on an auxiliary choice of a generic geodesic $\langle a,b\rangle\in \G$. In Lemma \ref{lem: independent of a,b}, we show that the function $\psi$ is independent of this auxiliary choice.

\begin{figure}[htbp]
    \centering
        \includegraphics[width=0.8\textwidth]{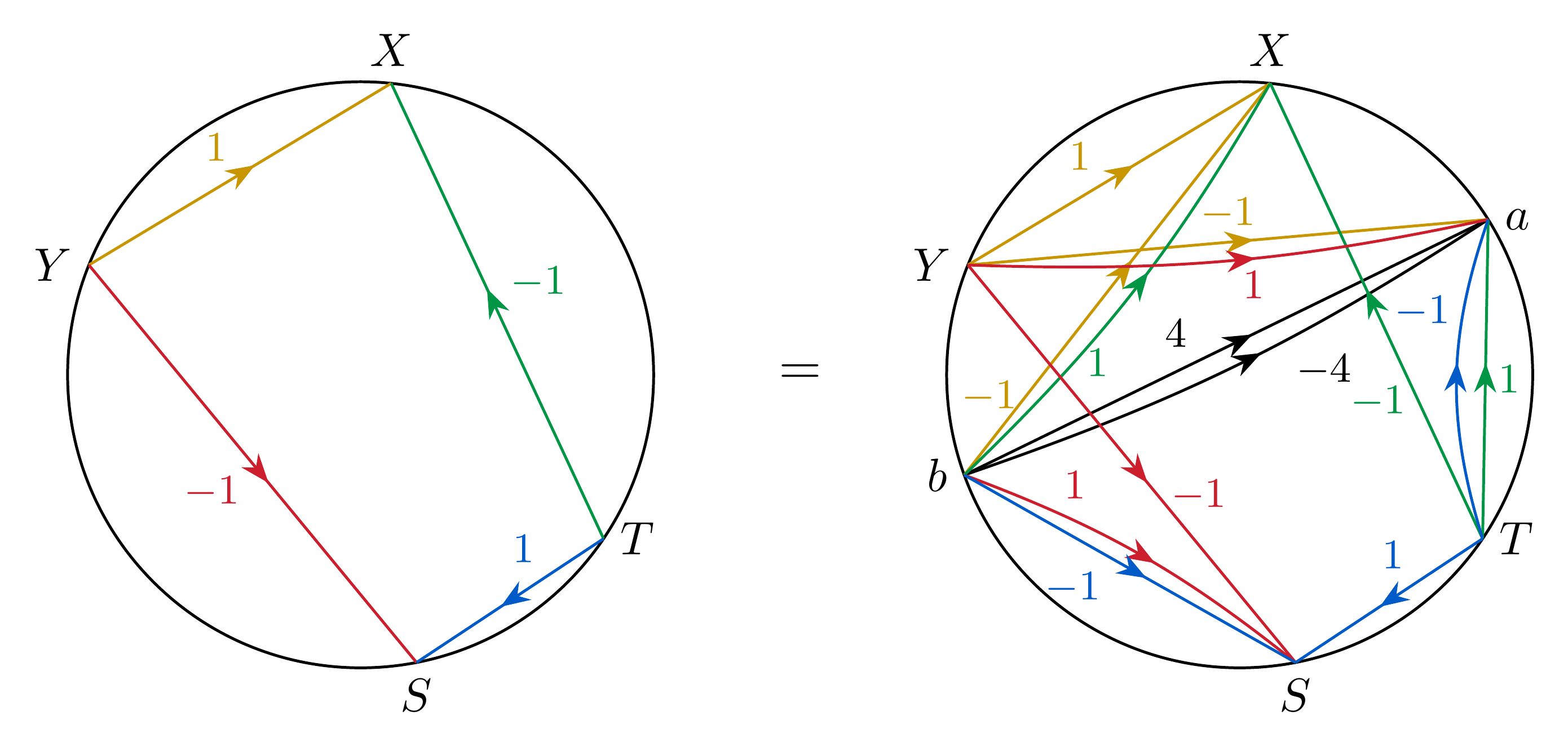}
        \caption{The decomposition of $\frac{\langle X,T\rangle \langle S,Y \rangle}{\langle X,Y \rangle \langle S,T\rangle}$ into four cross-ratios $\frac{\langle X,b\rangle \langle a,Y \rangle}{\langle X,Y \rangle \langle a,b\rangle}$, $\frac{\langle a,T\rangle \langle S,b \rangle}{\langle a,b \rangle \langle S,T\rangle}$, $\frac{\langle X,T\rangle \langle a,b \rangle}{\langle X,b \rangle \langle a,T\rangle}$, and  $\frac{\langle a,b\rangle \langle S,Y \rangle}{\langle a,Y \rangle \langle S,b\rangle}$.}
        \label{fig:diagram_decomp}
\end{figure}

We remark that there exist $[\vec{Z},\vec{z};\sigma] = [\vec{W},\vec{w};\tau]$ with $(\vec{Z},\vec{z},\sigma) \neq (\vec{W},\vec{w},\tau)$. Hence, a multifraction need not be expressed uniquely in terms of the data $(\vec{Z},\vec{z},\sigma)$. One could construct a similar map $\psi$ using the data $(\vec{Z},\vec{z},\sigma)$, but verifying that $\psi$ is well-defined in this context would be cumbersome.

By the universal property of free modules, we can extend $\psi$ to an $\R[\GG,c]$-module homomorphism $\Psi:\Swap\to\B/I$. In Lemma \ref{lem: algebra map}, we verify that $\Psi$ is an $\R$-algebra map. Finally, in Lemma \ref{lem: generators}, we show that the cross-ratios, together with $\G$ and $c$, generate $\Swap$ and $\B/I$ as $\R$-algebras in their respective contexts. The theorem follows by verifying that $\Psi$ is the inverse of $\hat\pi$ on these generators.

\section{Proof of Main Theorem}
We now give a proof of the main theorem. As described in Section \ref{subsection: strategy}, we will first identify a canonical basis for $\Swap$. We begin with the basic properties of multifractions. 

Recall from the previous section, that $\A$ is the ring $\R[S^1\times S^1]/J$, where $J$ is the ideal generated by $\langle x,x\rangle$ for all $x\in S^1$. Note that the ring $\A$ is isomorphic as a ring to $\R[\GG]$. A \emph{fraction of monomials} in $F(\A)$ is an element of the form $F = \prod_{x\neq y} \langle x,y\rangle^{e(x,y)}$, where $e$ is a finitely supported function on $\GG = S^1\times S^1-\Delta$ with values in $\Z$. Note that every fraction of monomials can be expressed uniquely as $F = 1$ or $F = N/D$ in lowest terms, where $N$ and $D$ are monomials without a common factor. We call such an expression reduced. A \emph{balanced} fraction of monomials is one with 
\begin{align*}
    \sum_{y\in S^1} e(x,y) = 0 \qquad \text{ and } \qquad \sum_{x\in S^1} e(x,y) = 0
\end{align*}
for all $x,y\in S^1$.

\begin{lemma} \label{lem: ppty of multifraction}
    \begin{enumerate}
        \item The product of two multifractions is a multifraction. Hence, $\A^*$ is spanned over $\R$ by the multifractions.
        \item Every nonzero multifraction is a balanced fraction of monomials.
        \item Let $F=N/D$ be a balanced fraction of monomials in its reduced expression, with $D = \prod_{i=1}^n \langle Y_i,y_i \rangle$, listing factors with multiplicity. Then there exists $\sigma \in S_n$ such that $N = \prod_{i=1}^n \langle Y_i,y_{\sigma(i)}\rangle$, and hence $F = [\vec{Y},\vec{y};\sigma]$.  
    \end{enumerate}
\end{lemma}
\begin{proof}
    \begin{enumerate}
        \item Let $[\vec{Z},\vec{z};\sigma], [\vec{W},\vec{w};\tau]$ be multifractions in $\A^*$. Then $[\vec{Z},\vec{z};\sigma][\vec{W},\vec{w};\tau] = [\vec{Z}\oplus \vec{W},\vec{z}\oplus\vec{w};\sigma\oplus \tau],$
        where 
        \begin{align*}
            \vec{Z}\oplus \vec{W} = (Z_1,\dots,Z_n,W_1,\dots,W_m), \qquad \vec{z}\oplus \vec{w} = (z_1,\dots,z_n,w_1,\dots,w_m),
        \end{align*} and $\sigma\oplus \tau\in S_{n+m}$ is given by $$(\sigma\oplus \tau)(i) = \begin{cases}
            \sigma(i) &, i\leq n, \\
            \tau(i-n)+n &, i>n.
        \end{cases}$$
        \item Fix $p \in S^1$. The number of factors with $p$ in their first slot is given by $|\{i\mid Z_i = p\}|$, both in the numerator and in the denominator. The number of factors with $p$ in their second slot is given by $|\{i\mid z_i = p\}|$ in the denominator and $|\{i|z_{\sigma(i)}=p\}|$ in the numerator. Since $\sigma$ is a permutation, the balance condition is satisfied.
        \item Write $N = \prod_{i=1}^m \langle X_i,x_i\rangle$ with multiplicity. Since $N/D$ is the reduced expression, the positive and negative parts of $e$ are exactly the exponents of $N$ and $D$, respectively. The balance condition implies that $(X_1,\dots,X_m) = (Y_1,\dots,Y_n)$ and $(x_1,\dots,x_m) = (y_1,\dots,y_n)$ as multisets. Hence, $m = n$. Choose permutations $\lambda,\tau \in S_n$ such that $X_i = Y_{\lambda(i)}$ and $x_i = y_{\tau(i)}$. For $\sigma = \tau\circ\lambda^{-1}$,
        \begin{align*}
            \prod_{i=1}^n\langle Y_i,y_{\sigma(i)}\rangle = \prod_{i=1}^n\langle Y_{\lambda(i)},y_{\tau(i)}\rangle = \prod_{i=1}^n \langle X_i,x_i\rangle = N, 
        \end{align*}
        which implies $F = [\vec{Y},\vec{y};\sigma]$.
    \end{enumerate}
\end{proof}
In view of the preceding lemma, we also call an expression $[\vec{Z},\vec{z};\sigma]$ reduced if $\prod_{i=1}^n\frac{\langle Z_i,z_{
\sigma(i)}\rangle}{\langle Z_i,z_i \rangle}$ is reduced. Let $T$ be the set of all nonzero multifractions in $\Swap$. By Lemma \ref{lem: ppty of multifraction}, $T$ is also the set of balance fractions of monomials. Also, by writing $N/D\in T$ in its reduced expression, we have
\[ 
N/D= \frac{\prod_{i=1}^s (\langle X_i,x_i \rangle)^{n_i}}{\prod_{j=1}^l (\langle Y_j,y_j \rangle)^{m_j}} \] 
where $X_i\neq x_i, Y_j\neq y_j, \langle X_i,x_i \rangle \neq \langle Y_j,y_j \rangle$, for all $i,j$, and we will usually write elements of $T$ in this reduced form. 

The next lemma shows that the set $T$ is an $\R[\GG,c]$-basis for $\Swap$, which we will use to construct the inverse $\R[\GG,c]$-linear map from $\Swap$ to $\B/I$.   
\begin{lemma} \label{lem: basis for Swap}
    The set $T$ is an $\R[\GG,c]$-basis for $\Swap$.    
\end{lemma}
\begin{proof}
    Note that $\Swap = \A^*[L_\GG,c] \cong \R[\GG,c]\otimes_\R \A^*$. Furthermore, if $T$ is an $\R$-basis for $\A^*$, then $1\otimes T = \{1\otimes t|t\in T\}$ is an $\R[\GG,c]$-basis for $\R[\GG,c]\otimes_\R \A^*$. Therefore, the identification between $\Swap$ and $\R[\GG,c]\otimes_\R \A^*$ shows that $T$ is an $\R[\GG,c]$-basis for $\Swap$.
   
    Hence, it suffices to show that $T$ is an $\R$-basis for $\A^*$. Since $\A^*$ is spanned by multifraction, by Lemma \ref{lem: ppty of multifraction}, $\A^*$ is spanned by $T$.
   
    For linear independence, suppose $\sum_i a_i t_i = 0$ with distinct $t_i$. Write $t_i = \frac{H_i}{K_i}$ in reduced form where $H_i$ and $K_i$ are monomials. Hence, multiplying by $\prod_{j}K_j$,
    \begin{align*}
        \sum_i a_i (\prod_{j\neq i} K_j) H_i = 0.
    \end{align*}
    We claim that $(\prod_{j\neq i} K_j) H_i \neq (\prod_{j\neq s} K_j) H_s$, for all $i\neq s$. Suppose, for a contradiction, that $(\prod_{j\neq i} K_j) H_i = (\prod_{j\neq s} K_j) H_s$ for some $i\neq s$. Canceling the common term $\prod_{j\neq i,s}K_j$ gives $K_sH_i = K_iH_s$, that is, $t_i = H_i/K_i = H_s/K_s = t_s$, contradicting $t_i \neq t_s$. Since distinct monomials in a polynomial ring are linearly independent, $a_i = 0$ for all $i$.
\end{proof}

The following lemma, which follows from a direct computation, decomposes multifractions into products of cross-ratios.

\begin{lemma} \label{lem:presentation of T multifraction}
    For all $[\vec{Z},\vec{z};\sigma] \in T$, let $\langle a,b\rangle\in \G$ be such that $a,b \neq Z_i,z_i$  for all $i$. Then,
    $$[\vec{Z},\vec{z};\sigma] = \prod_{i=1}^n \frac{\langle Z_i,z_{\sigma(i)} \rangle,\langle a,b \rangle}{\langle Z_i,b \rangle,\langle a,z_{\sigma(i)} \rangle}\frac{ \langle Z_i,b \rangle\langle a,z_i \rangle}{ \langle Z_i,z_i \rangle,\langle a,b \rangle}.$$
\end{lemma}
Lemma \ref{lem:presentation of T multifraction} gives a way to decompose a multifraction into a product of cross-ratios. Note that the reduced expression of a balance fraction of monomials is preferable, since it is unique, whereas there exist $(\vec{Z},\vec{z},\sigma)\neq (\vec{W},\vec{w},\tau)$ such that $[\vec{Z},\vec{z};\sigma] = [\vec{W},\vec{w};\sigma]$ with both expressions reduced. Therefore, in terms of the reduced expression of an element $N/D\in T$, we have the following observation.
\begin{lemma} \label{lem:presentation of T}
    For every $N/D \in T$ in its reduced expression, with 
    \begin{align*}
        N = \prod_{i=1}^s (\langle X_i,x_i \rangle)^{n_i}, \qquad D = \prod_{j=1}^l (\langle Y_j,y_j \rangle)^{m_j},
    \end{align*} 
    let $\langle a,b\rangle\in \G$ be such that $a,b \neq X_i,x_i,Y_j,y_j$  for all $i,j$. Then,
    $$\frac{\prod_{i=1}^s (\langle X_i,x_i \rangle)^{n_i}}{\prod_{j=1}^l (\langle Y_j,y_j \rangle)^{m_j}} = \prod_{i=1}^s \frac{(\langle X_i,x_i \rangle,\langle a,b \rangle)^{n_i}}{(\langle X_i,b \rangle,\langle a,x_i \rangle)^{n_i}}\prod_{j=1}^l\frac{ (\langle Y_j,b \rangle\langle a,y_j \rangle)^{m_j}}{ (\langle Y_j,y_j \rangle,\langle a,b \rangle)^{m_j}}.$$
\end{lemma}
\begin{proof}
    We first choose a reduced expression $[\vec{Z},\vec{z};\sigma] = \frac{\prod_{i=1}^s (\langle X_i,x_i \rangle)^{n_i}}{\prod_{j=1}^l (\langle Y_j,y_j \rangle)^{m_j}}$. Since this expression is reduced, as $i$ runs from $1$ to $n$, the factors $\langle Z_i,z_{\sigma(i)}\rangle$ runs through the factors of $N$ and $\langle Z_i,z_i\rangle$ runs through the factors of $D$, counted with multiplicity. Regrouping the product in Lemma \ref{lem:presentation of T multifraction} gives the formula.
\end{proof}

Let $\psi:T \to \B/I$ be the map given by $$\frac{\prod_{i=1}^s (\langle X_i,x_i \rangle)^{n_i}}{\prod_{j=1}^l (\langle Y_j,y_j \rangle)^{m_j}} \mapsto \prod_{i=1}^s\lceil \langle X_i, b\rangle, \langle a,x_i\rangle \rceil^{n_i} \prod_{j=1}^l\lceil \langle Y_j, y_j\rangle,\langle a , b\rangle \rceil^{m_j}$$ for any $\langle a,b \rangle \in \G$ with $a,b \neq X_i,x_i,Y_j,y_j$ for all $i,j$, where $\frac{\prod_{i=1}^s (\langle X_i,x_i \rangle)^{n_i}}{\prod_{j=1}^l (\langle Y_j,y_j \rangle)^{m_j}}$ is the reduced expression. 

\begin{lemma} \label{lem: independent of a,b}
    The map $\psi$ is well-defined, i.e. independent of the choice of $a,b$.
\end{lemma}
\begin{proof}
    Let $N/D\in T$ be in reduced form, and $\langle a,b\rangle,\langle a',b'\rangle \in \G$ satisfy the condition above. By Lemma \ref{lem: ppty of multifraction}, write $N/D = [\vec{Z},\vec{z};\sigma]$ as a reduced expression. Assume that $\sigma = (1,2,\dots,n)$ is the standard $n$-cycle. Then,
    \begin{align*}
        \psi(N/D) &= \prod_{i=1}^n \lceil \langle Z_i,b \rangle, \langle a,z_{i+1} \rangle \rceil \lceil \langle Z_i,z_i \rangle, \langle a,b \rangle \rceil \\
        &\stackrel{\eqref{rel 5}}{=} \prod_{i=1}^n \lceil \langle Z_i,z_i \rangle,\langle a,z_{i+1} \rangle, \langle a,b \rangle \rceil \\
        &= \prod_{i=1}^n \lceil \langle a,b \rangle, \langle Z_i,z_i \rangle,\langle a,z_{i+1} \rangle \rceil \\
        &\stackrel{\eqref{rel 3}}{=} \prod_{i=1}^n \lceil \langle Z_i,z_i \rangle,\langle a,z_{i+1} \rangle \rceil.
    \end{align*}
    If $n = 2$,
    \begin{align*}
        \lceil \langle Z_1,z_1 \rangle,\langle a,z_{2} \rangle \rceil\lceil \langle Z_2,z_2 \rangle,\langle a,z_{1} \rangle \rceil
        &= \lceil \langle a,z_{2} \rangle, \langle Z_1,z_1 \rangle \rceil \lceil \langle a,z_{1} \rangle , \langle Z_2,z_2 \rangle \rceil  \\
        &\stackrel{\eqref{rel 5}}{=} \lceil \langle a,z_1 \rangle,  \langle Z_1,z_1 \rangle, \langle Z_2,z_2 \rangle \rceil.
    \end{align*}
    Hence,  
    \[ \lceil \langle a,z_{1} \rangle,  \langle Z_1,z_1 \rangle, \langle Z_2,z_2 \rangle \rceil
    \stackrel{\eqref{rel 4}}{=}\lceil \langle Z_1,z_1 \rangle,\langle Z_2,z_{2} \rangle \rceil. \]
    Note that the final expression is independent of the points $a,b$. If $n>2$, by induction and relations \ref{rel 5} and \ref{rel 4},
    \begin{align*}
        \prod_{i=1}^n \lceil \langle Z_i,z_i \rangle,\langle a,z_{i+1} \rangle \rceil &=
        \lceil \langle a,z_n \rangle,\langle Z_{1},z_{1} \rangle,\dots, \langle Z_{n-1},z_{n-1} \rangle \rceil \lceil \langle a,z_{1} \rangle,\langle Z_n,z_n \rangle \rceil \\
        &\stackrel{\eqref{rel 5}}{=} \lceil \langle a,z_1\rangle,\langle Z_1,z_1 \rangle, \dots, \langle Z_{n-1},z_{n-1} \rangle, \langle Z_n,z_{n} \rangle \rceil \\
        &\stackrel{\eqref{rel 4}}{=} \lceil \langle Z_1,z_1 \rangle, \dots, \langle Z_{n-1},z_{n-1} \rangle, \langle Z_n,z_{n} \rangle \rceil.
    \end{align*}
   
    For a general permutation $\sigma$, we decompose $\sigma$ into disjoint cycles $c_i$, i.e. $\sigma = c_1\dots c_k$. Then $[\vec{Z},\vec{z};\sigma] = \prod_{i=1}^k[\vec{Z}_{c_i},\vec{z}_{c_i};(1,2,\dots,n_i)]$, where $\vec{Z}_{c_i} = (Z_{c_i(1)}, Z_{c_i(2)},\dots,Z_{c_i(n_i)} )$ and $\vec{z}_{c_i}$ is defined similarly. Therefore,
    \begin{align*}
        \psi\left(\prod_{i=1}^k[\vec{Z}_{c_i},\vec{z}_{c_i};(1,2,\dots,n_i)]\right) =
        \prod_{i=1}^k \left(\prod_{j=1}^{n_i}\lceil \langle Z_{c_i(j)},b \rangle, \langle a,z_{{c_i(j+1)}} \rangle \rceil \lceil \langle Z_{c_i(j)},z_{c_i(j)} \rangle, \langle a,b \rangle \rceil\right).
    \end{align*}
 Notice that the term on the right-hand side now corresponds to the previous case of the standard $n$-cycle.
\end{proof}

Since $T$ is an $\R[\GG,c]$-basis for $\Swap$, we can extend $\psi$ uniquely to an $\R[\GG,c]$-module homomorphism $\Psi:\Swap \to \B/I$.
\begin{lemma} \label{lem: algebra map}
    The map $\Psi$ is an $\R$-algebra map.
\end{lemma}
\begin{proof}
    It suffices to show that $\psi(xy) = \psi(x)\psi(y)$ for all $x,y \in T$. Suppose $x = \frac{C}{D} \frac{M}{N}, y = \frac{E}{F} \frac{N}{M}$, where
    \begin{align*}
        C = \prod_{i=1}^{n_1} \langle \hat{c}_i,c_i\rangle & \qquad D = \prod_{i=1}^{n_2} \langle \hat{d}_i,d_i\rangle \\
        E = \prod_{i=1}^{n_3} \langle \hat{e}_i,e_i\rangle & \qquad F = \prod_{i=1}^{n_4} \langle \hat{f}_i,f_i\rangle \\        M = \prod_{i=1}^{s} \langle \hat{m}_i,m_i\rangle & \qquad N = \prod_{i=1}^{t} \langle \hat{n}_i,n_i\rangle
    \end{align*}
    such that $xy = \frac{CE}{DF}$ has no common factor in both the numerator and denominator. Choose $\langle a,b \rangle\in \G$ such that $a$ and $b$ are not equal to the endpoints of the geodesics in $C,D,E,F,M,N$. Then,
    \begin{align*}
        \psi(x) \psi(y) &= \prod_{i=1}^{n_1} \lceil \langle a, c_i\rangle , \langle \hat c_i, b\rangle  \rceil
        \prod_{i=1}^{n_2} \lceil \langle \hat d_i, d_i \rangle,\langle a,b \rangle \rceil
        \prod_{j=1}^{n_3} \lceil \langle a, e_j\rangle , \langle \hat e_j, b\rangle  \rceil
        \prod_{j=1}^{n_4}\lceil \langle \hat f_j, f_j \rangle,\langle a,b \rangle \rceil \\
        &\prod_{k=1}^s \lceil \langle a,m_k  \rangle, \langle \hat m_k,b\rangle \rceil \lceil \langle \hat m_k,m_k  \rangle, \langle a,b\rangle \rceil \prod_{l=1}^t \lceil \langle a,n_l  \rangle, \langle \hat n_l,b\rangle \rceil \lceil \langle \hat n_l, n_l \rangle, \langle a,b\rangle \rceil \\
        &= \psi(xy) \prod_{k=1}^s \lceil \langle a,b\rangle, \langle \hat m_k,m_k  \rangle \rceil \lceil \langle a,m_k  \rangle, \langle \hat m_k,b\rangle \rceil
        \prod_{l=1}^t \lceil \langle a,b\rangle, \langle \hat n_l,n_l  \rangle \rceil \lceil\langle a,n_l  \rangle, \langle \hat n_l,b\rangle \rceil \\
        &\stackrel{\eqref{rel 5}}{=} \psi(xy)\prod_{k=1}^s \lceil \langle a,m_k  \rangle,\langle \hat m_k,m_k  \rangle, \langle \hat m_k,b\rangle \rceil
        \prod_{l=1}^t \lceil \langle a,n_l  \rangle,\langle \hat n_l,n_l  \rangle, \langle \hat n_l,b\rangle \rceil \\
        &\stackrel{\eqref{rel 4}}{=} \psi(xy)\prod_{k=1}^s \lceil \langle \hat m_k,m_k  \rangle, \langle \hat m_k,b\rangle \rceil \prod_{l=1}^t \lceil \langle \hat n_l,n_l  \rangle, \langle \hat n_l,b\rangle \rceil \\
        &\stackrel{\eqref{rel 2}}{=} \psi(xy).
    \end{align*}
\end{proof}

\begin{lemma}\label{lem: generators}
    The algebra $\Swap$ is generated by $\frac{\langle X,y\rangle \langle Y,x \rangle}{\langle X,x\rangle, \langle Y,y \rangle}$, $l_g$ and $c$, and $\B/I$ is generated by $\lceil \langle X,x\rangle, \langle Y,y \rangle \rceil$, $g$ and $c$ as algebras, where $X, x , Y,y$ are pairwise distinct in both cases.
\end{lemma}
\begin{proof}
    Lemma \ref{lem:presentation of T} shows the former. Let $\lceil  \langle Z_1,z_1\rangle,\dots,\langle Z_n,z_n\rangle \rceil \in \B/I$. By an argument similar to that of Lemma \ref{lem: independent of a,b}, $$\lceil\langle Z_1,z_1\rangle,\dots,\langle Z_n,z_n\rangle \rceil = \prod_{i=1}^n \lceil \langle Z_i,b \rangle, \langle a,z_{i+1} \rangle \rceil \lceil \langle Z_i,z_i \rangle, \langle a,b \rangle \rceil$$
    for $\langle a, b \rangle \in \G$ with  $a,b\neq Z_j,z_j$ for all $j=1,\cdots,n$.
    Note that $Z_i\neq z_i$ for all $i=1,\dots,n$. If $Z_i = z_{i+1}$ for some $i$, then by relation \ref{rel 1}, $\lceil \langle Z_i,b \rangle, \langle a,z_{i+1} \rangle \rceil = 0$.
\end{proof}

\begin{proof}[Proof of Theorem \ref{thm: relationship between G and S}]
    Note that the generators of $I$ are in the kernel of $\pi$. This induces a map $\hat\pi:\B/I \to \Swap$. It suffices to show that $\Psi$ and $\hat \pi$ are inverses of one another on the generators in Lemma \ref{lem: generators}. For the generators $c, g, \lceil \langle X,x\rangle,\langle Y,y\rangle\rceil \in \B/I$ with $X,x,Y,y$ pairwise distinct, we have $\Psi(\hat \pi(c)) = \Psi(c) = c, \Psi(\hat \pi(g)) = \Psi(l_g) = g,$  and, by an argument similar to that in Lemma \ref{lem: independent of a,b}    
    \begin{align*}
        \Psi(\hat \pi(\lceil \langle X,x\rangle,\langle Y,y\rangle\rceil)) &= \Psi\left(\frac{\langle X,y \rangle\langle Y,x \rangle}{\langle X,x \rangle\langle Y,y \rangle}\right) \\
        &= \lceil \langle X,x\rangle,\langle a,b\rangle\rceil
        \lceil \langle Y,y\rangle,\langle a,b\rangle\rceil
        \lceil \langle a,y\rangle,\langle X,b\rangle\rceil
        \lceil \langle a,x\rangle,\langle Y,b\rangle\rceil\\
        &= \lceil \langle a,z_2\rangle,\langle Z_1,b\rangle\rceil
        \lceil \langle Z_1,z_1\rangle,\langle a,b\rangle\rceil
        \lceil \langle a,z_1\rangle,\langle Z_2,b\rangle\rceil
        \lceil \langle Z_2,z_2\rangle,\langle a,b\rangle\rceil \\
        &= \lceil \langle X,x\rangle,\langle Y,y\rangle\rceil
    \end{align*}
    where $(Z_1,Z_2) = (X,Y)$, $(z_1,z_2) = (x,y)$. Conversely, $\hat\pi(\Psi(c)) = c, \hat\pi(\Psi(l_g)) = \hat\pi(g) = l_g,$ and computation above shows $\hat\pi\circ\Psi$ fixes cross-ratios. Since $\hat \pi$ and $\Psi$ are algebra maps, $\hat\pi\circ\Psi$ is the identity map on $\Swap$.
\end{proof}

\printbibliography
\end{document}